\documentclass[reqno, a4paper, 11pt]{amsart}
\usepackage{amsmath,amsthm,amssymb}
\usepackage{graphicx}
\usepackage{enumitem}
\usepackage{time} 
\usepackage{comment}
\usepackage{thm-restate}

\usepackage[dvipdfmx]{hyperref}
\hypersetup{
	setpagesize=false,
	bookmarksnumbered=true,%
	bookmarksopen=true,%
	colorlinks=true,%
	linkcolor=magenta,
	citecolor=blue,
}
\theoremstyle{plain}
\newtheorem{theorem}{Theorem}[section]
\newtheorem{lemma}[theorem]{Lemma}

\newtheorem{proposition}[theorem]{Proposition}
\theoremstyle{definition}
\newtheorem{definition}[theorem]{Definition}
\newtheorem{remark}[theorem]{Remark}
\newcommand{\CC}{\mathbb{C}}
\newcommand{\calT}{\mathcal{T}}
\newcommand{\calRT}{\mathcal{T}_\mathrm{rel}}
\newcommand{\calD}{\mathcal{D}}
\newcommand{\calO}{\mathcal{O}}
\newcommand{\Int}{\mathrm{int\,}}
\newcommand{\bcs}{\,\natural\,}
\newcommand{\rtdcap}[1]{\mathrm{Cap}\left( #1 \right)}
\DeclareMathOperator{\id}{id}

\title[Distinguishing diffeomorphism types of relative trisections]{Distinguishing diffeomorphism types of relative trisections}
\author[Natsuya Takahashi]{Natsuya Takahashi}
\date{February 15, 2026. \textit{Revised}: June 19, 2026.}
\address{Department of Mathematics, College of Science and Technology, Nihon University, 1-8-14, Kanda-Surugadai, Chiyoda-ku, Tokyo 101-8308, Japan}
\email{takahashi.natsuya@nihon-u.ac.jp}
\begin{document}
\begin{abstract}
We distinguish diffeomorphism types of relative trisections using a ``capping'' operation, which yields a trisection diagram of a closed 4-manifold from a relative trisection diagram. Using this operation, we give various examples of non-diffeomorphic relative trisections of the same 4-manifold with boundary. We also study how the corresponding trisected 4-manifold changes under the capping operation.
\end{abstract}
\maketitle

\section{Introduction}

A trisection of a closed $4$-manifold $X$ is a decomposition $X = X_{1} \cup X_{2} \cup X_{3}$ into three $1$-handlebodies whose triple intersection is a surface. 
A similar decomposition for a $4$-manifold with connected boundary is called a relative trisection.
These notions were introduced by Gay and Kirby~\cite{GayKir16} as $4$-dimensional analogues of Heegaard splittings for $3$-manifolds.
Just as a Heegaard splitting is encoded by a Heegaard diagram, a trisection is represented by a trisection diagram, which is a $4$-tuple $(\Sigma; \alpha, \beta, \gamma)$ consisting of a surface and three collections of curves.


A basic invariant of a Heegaard splitting is its genus, which is defined as the genus of the splitting surface.
A trisection of a closed $4$-manifold is associated with a pair $(g,k)$, where $g$ is the genus of the triple intersection surface and $k$ is the genus of each sector.
Similarly, a relative trisection of a $4$-manifold with boundary is associated with a $4$-tuple $(g,k; p,b)$, where the additional parameters $p$ and $b$ are the genus of the induced open book and the number of the binding components, respectively.


One of the themes in trisection theory is the classification of trisections up to diffeomorphism.
Two (relative) trisections $X = X_{1} \cup X_{2} \cup X_{3}$ and $X' = X'_{1} \cup X'_{2} \cup X'_{3}$ are called \textit{diffeomorphic} if there exists an orientation-preserving diffeomorphism $f \colon X \to X'$ such that $f(X_{i}) = X'_{i}$ for each $i \in \{1,2,3\}$. 
In this paper, we focus on $4$-manifolds admitting more than one diffeomorphism type of trisections.
It is immediate that if $(g,k)\neq (g',k')$, then two trisections with these parameters are not diffeomorphic.
This observation naturally leads to the question of the non-uniqueness of trisections with the same parameters.
Specifically, we consider the following problems:
\begin{itemize}
\item
In the closed case, does a given closed $4$-manifold admit more than one $(g,k)$-trisection up to diffeomorphism?
\item
In the relative case, does a given $4$-manifold with boundary admit more than one $(g,k;p,b)$-relative trisection up to diffeomorphism?
\end{itemize}


These problems have been studied primarily for trisections of minimal genus.
Here, a trisection is said to have \textit{minimal genus} if the underlying $4$-manifold admits no trisection of smaller genus. This minimal value, denoted by $g(X)$, is called the \textit{trisection genus} of $X$.
In the closed case, Islambouli~\cite{Isl21} proved that certain spun Seifert fibered spaces admit non-diffeomorphic minimal genus trisections with the same parameters, using Nielsen equivalence to distinguish them (see also \cite{IsoOga24a} for related work).
In the relative case, the author~\cite{Tak24} gave a pair of non-diffeomorphic minimal genus relative trisections of $S^2\times D^2$.
The proof in the paper is based on a diagrammatic approach.
Recall that two (relative) trisections are diffeomorphic if and only if their associated (relative) trisection diagrams are diffeomorphism and handleslide equivalent (see Subsection~\ref{sub:trisdiag}).
To distinguish relative trisections, the author used an operation that yields a trisection diagram of a closed $4$-manifold from a relative trisection diagram.

\begin{definition}\label{dfn:gdo}
Let $\calD=(\Sigma;\alpha,\beta,\gamma)$ be a $(g,k;p,b)$-relative trisection diagram.
Let $\rtdcap{\Sigma}$ be the genus-$g$ closed surface obtained from $\Sigma$ by gluing $b$ disks to the boundary circles.
Then, we define the capping of $\calD$ as $\rtdcap{\calD}:=(\rtdcap{\Sigma};\alpha,\beta,\gamma)$, where the $\alpha$-, $\beta$-, and $\gamma$-curves are regarded as lying on $\rtdcap{\Sigma}$.
\end{definition}
\begin{figure}[!htbp]
\centering
\includegraphics[scale=0.7]{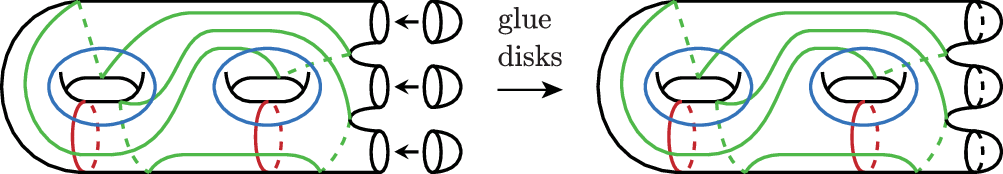}
\caption{Left: relative trisection diagram $\calD=(\Sigma;\alpha,\beta,\gamma)$. Right: trisection diagram $\rtdcap{\calD}=(\rtdcap{\Sigma};\alpha,\beta,\gamma)$.}
\label{fig:td-gd}
\end{figure}

See Figure~\ref{fig:td-gd} for an example of the operation of Definition~\ref{dfn:gdo}.
The next theorem gives a useful tool for distinguishing relative trisection diagrams.

\begin{theorem}[\cite{Tak24}]\label{thm:equiv-rtd-gd}
If $\calD$ is a $(g,k;0,b)$-relative trisection diagram, then $\rtdcap{\calD}$ is a $(g,k-b+1)$-trisection diagram of a closed $4$-manifold.
In addition, if two $(g,k;0,b)$-relative trisection diagrams $\calD$ and $\calD'$ are diffeomorphism and handleslide equivalent, then $\rtdcap{\calD}$ and $\rtdcap{\calD'}$ are also diffeomorphism and handleslide equivalent.
\end{theorem}

By this theorem, if the closed trisection diagrams $\rtdcap{\calD}$ and $\rtdcap{\calD'}$ are not diffeomorphism and handleslide equivalent, then the original relative trisection diagrams $\calD$ and $\calD'$ are also not equivalent.
In~\cite{Tak24}, we used this method to distinguish diffeomorphism types of relative trisections for $S^2\times D^2$.
The present paper gives further examples using this approach.
Our first result is the following.


\begin{theorem}\label{thm:cl+1hdlbd}
Suppose that a closed $4$-manifold $X$ admits non-diffeomorphic $(g,k)$-trisections.
Then, for any non-negative integer $n$, the $4$-manifold $X\# (\natural^n{S^1\times D^3})$ admits non-diffeomorphic $(g,k;0,n+1)$-relative trisections whose induced open books are equivalent.
\end{theorem}

This theorem follows from taking the connected sums of non-diffeomorphic trisections of a closed $4$-manifold with a genus-$0$ relative trisection of the $1$-handlebody $\natural^{n}(S^{1} \times D^{3})$. 
(For details on the connected sum operation between closed and relative trisections, see Appendix~A.)

A possible approach to distinguishing diffeomorphism types of relative trisections is to compare their induced open books.
It is worth emphasizing that the relative trisections appearing in Theorem~\ref{thm:cl+1hdlbd} are not diffeomorphic, although their induced open books are equivalent.
Furthermore, in the case $n=0$ of Theorem~\ref{thm:cl+1hdlbd}, we obtain the following refinement.

\begin{theorem}\label{thm:-D4}
If a closed $4$-manifold $X$ admits non-diffeomorphic $(g,k)$-trisections of minimal genus, 
then $X-\Int{D^4}$ admits non-diffeomorphic $(g,k;0,1)$-relative trisections of minimal genus.
\end{theorem}

This theorem is obtained by combining the case $n=0$ of Theorem~\ref{thm:cl+1hdlbd} with the next proposition.

\begin{proposition}\label{thm:tg-eq}
For any closed $4$-manifold $X$, the trisection genus of $X$ and that of $X-\Int{D^4}$ coincide.
\end{proposition}

Although this statement may appear straightforward at first glance, its proof is not trivial.
The inequality $g(X) \geq g(X-\Int{D^4})$ is easy to check, but the reverse inequality requires a more subtle argument.


We also give non-diffeomorphic relative trisections for the same $4$-manifold that cannot be realized as the connected sum of a closed $4$-manifold and a $1$-handlebody.

\begin{theorem}\label{thm:ndrt-inf}
%
For any positive integer $n$, there exists a $4$-manifold $W$ with boundary such that the following conditions hold.
\begin{itemize}
\item $W$ admits two genus-$2n$ relative trisections $\calRT$ and $\calRT'$ with the same parameters.
\item The open books induced by $\calRT$ and $\calRT'$ are equivalent.
\item $\calRT$ and $\calRT'$ are not diffeomorphic as trisections.
\item The trisection genus of $W$ is $2n$.
\item $W$ cannot be decomposed as the connected sum
of a closed $4$-manifold and a $1$-handlebody.
\end{itemize}
\end{theorem}

As an example satisfying the conditions of this theorem, we take the boundary connected sum of $n$ copies of $S^2 \times D^2$ (see Subsection~\ref{subsec:infinite}).
%


Furthermore, we obtain a result on non-diffeomorphic relative trisections for a contractible $4$-manifold known as the Akbulut cork.

\begin{theorem}\label{thm:ndrt-Ac}
The Akbulut cork admits non-diffeomorphic $(3,3;0,4)$-relative trisections, which are of minimal genus.
\end{theorem}


By dropping the conditions on the minimality of genus and the equivalence of induced open books, we obtain the following result.

\begin{theorem}\label{thm:non-min-genus}
For any $4$-manifold $W$ with boundary, there exist two relative trisections $\calRT$ and $\calRT'$ of $W$ such that their parameters $(g,k;p,b)$ are equal, but $\calRT$ and $\calRT'$ are not diffeomorphic.
\end{theorem}

Given an arbitrary relative trisection of a $4$-manifold $W$ with boundary, we provide a construction that yields two non-diffeomorphic relative trisections of $W$ with the same parameters. 
Note that the resulting relative trisections are not of minimal genus and do not induce equivalent open books.

In addition, we study the geometric meaning of the operation of Definition~\ref{dfn:gdo}.
More precisely, we describe how the corresponding trisected 4-manifold changes when the boundary of a relative trisection diagram is capped off.

\begin{theorem}\label{thm:cap-handle}
Let $\calD$ be a $(g, k; 0, b)$-relative trisection diagram, and let $\rtdcap{\calD}$ be the $(g,k-b+1)$-trisection diagram obtained by the operation of Definition~\ref{dfn:gdo}.
Let $W = W_1 \cup W_2 \cup W_3$ be the $(g, k; 0, b)$-relative trisection induced by $\calD$, and let $X = X_1 \cup X_2 \cup X_3$ be the $(g,k-b+1)$-trisection induced by $\rtdcap{\calD}$.
Then the closed $4$-manifold $X$ is obtained from $W$ by attaching $(b-1)$ $2$-handles and one $4$-handle, where the $2$-handles are attached along $(b-1)$ components of $\partial(W_1 \cap W_2 \cap W_3)$.
\end{theorem}

Furthermore, we explicitly determine the framings of these $2$-handles. For details, see the proof of Theorem~\ref{thm:cap-b>1}.
Since a relative trisection can be converted into a handle decomposition (see \cite{KimMil20}), Theorem~\ref{thm:cap-handle} gives a handle decomposition of the closed $4$-manifold corresponding to the capped diagram.


Finally, independently of the preceding arguments, we study the existence of minimal genus relative trisections of the same $4$-manifold with distinct parameters $(g,k;p,b)$, which are therefore non-diffeomorphic.

\begin{theorem}\label{thm:diff-gkpb}
There exists a $4$-manifold $W$ with boundary that admits two relative trisections $\calRT$ and $\calRT'$ satisfying the following conditions:
\begin{itemize}
\item
The genera of $\calRT$ and $\calRT'$ are minimal.
\item
$\calRT$ and $\calRT'$ have distinct parameters $(g,k;p,b)$. 
\end{itemize}
\end{theorem}

We explicitly give a pair of relative trisections for a contractible $4$-manifold that satisfies the conditions of Theorem~\ref{thm:diff-gkpb}.
To the best of our knowledge, this provides the first known example of such relative trisections.

\subsection*{Conventions}
All manifolds and surfaces are assumed to be compact, connected, oriented, and smooth.
All $4$-manifolds with boundary are assumed to have non-empty connected boundary.
If two smooth manifolds $X$ and $Y$ are orientation-preserving diffeomorphic, we write $X\cong Y$.
Let $\Sigma_g$ be a closed, connected, oriented surface of genus $g$, and let $\Sigma_{g,b}$ be a compact, connected, oriented surface of genus $g$ with $b$ boundary components.

\section{preliminaries}

\subsection{Trisections of closed $4$-manifolds}

In \cite{GayKir16}, trisections are defined using a model of a $4$-dimensional $1$-handlebody.
For integers $g$ and $k$ such that $0\leq k\leq g$, set $Z_{k}:=\natural^k(S^1\times D^3)$ and $Y_{k}:=\partial{Z_{k}}=\#^k(S^1\times S^2)$.
Let $Y_{k}=Y_{g,k}^-\cup Y_{g,k}^+$ be the standard genus-$g$ Heegaard splitting of $Y_{k}$.
Note that the connected sum of $S^1\times S^2$'s admits a unique Heegaard splitting in each genus (see Section~4 of \cite{Wald68}).

\begin{definition}[\cite{GayKir16}]\label{dfn:tris-model}
Let $X$ be a closed, connected, oriented, smooth $4$-manifold.
A decomposition $\calT:X=X_1\cup X_2\cup X_3$ is called a $(g,k)$-\textit{trisection} of $X$ if it satisfies the following conditions.
\begin{itemize}
\item
For each $i \in\{1,2,3\}$, there exists a diffeomorphism $\varphi_i:X_i\to Z_{k}$.
\item
For each $i \in\{1,2,3\}$, taking indices mod $3$,
\begin{equation*}
\varphi_i(X_i\cap X_{i+1})=Y^-_{g,k} \quad\textrm{and}\quad \varphi_i(X_i\cap X_{i-1})=Y^+_{g,k}.
\end{equation*}
\end{itemize}
\end{definition}

By Definition~\ref{dfn:tris-model}, if a decomposition $X=X_1\cup X_2\cup X_3$ is a $(g,k)$-trisection, then the following properties hold.
\begin{itemize}
\item
For each $i\in\{1,2,3\}$, the sector $X_i$ is diffeomorphic to the genus-$k$ $4$-dimensional $1$-handlebody $\natural^k(S^1\times D^3)$.
\item
For each $i,j\in\{1,2,3\}$ such that $i\neq j$, the double intersection $X_i\cap X_j(=\partial{X_i}\cap\partial{X_j})$ is diffeomorphic to the genus-$g$ $3$-dimensional $1$-handlebody $\natural^g(S^1\times D^2)$.
\item
The triple intersection $X_1\cap X_2\cap X_3$ is diffeomorphic to the genus-$g$ closed surface $\Sigma_{g}$.
\item
For each $i\in\{1,2,3\}$, we have $\partial{X_i}=(X_i\cap X_{i+1})\cup(X_i\cap X_{i-1})$, and this is a genus-$g$ Heegaard splitting of $\#^k(S^1\times S^2)$. 
\end{itemize}

Gay and Kirby proved that any closed, connected, oriented, smooth $4$-manifold admits a trisection \cite[Theorem~4]{GayKir16}.
As mentioned in the introduction, we henceforth assume that all closed $4$-manifolds are connected, oriented, and smooth.

\subsection{Relative trisections of $4$-manifolds with boundary}

In \cite{GayKir16}, Gay and Kirby also introduced the notion of a relative trisection for a $4$-manifold with connected boundary.
We now introduce the definition of a relative trisection rephrased by Castro, Gay, and Pinz\'{o}n-Caiced~\cite{CasGayPin18a}.
As in the closed case, the definition was given using a model $Z_{k}$ of a $4$-dimensional $1$-handlebody.
Let $g$, $k$, $p$, and $b$ be integers such that
\begin{align*}
{g,k,p\geq0,} \quad {b\geq1,} \quad\text{and}\quad {2p+b-1\leq k\leq g+p+b-1.}
\end{align*}
Remark that $g-p$, $g-k+p+b-1$, and $n:=k-2p-b+1$ are non-negative integers.
We will define $Z_{k}$ as a boundary connected sum of two $4$-manifolds $U_{p,b}$ and $V_{n}$.
We begin by constructing these constituent pieces.

First, we construct the $4$-manifold $U_{p,b}$ and give a decomposition of its boundary.
Let $\Delta$ be the third of the unit disk defined as
\begin{equation*}
\Delta := \left\{ re^{\sqrt{-1}\theta} \in \CC \mid {0\leq r \leq1}, \  {-\frac{\pi}{3} \leq \theta \leq \frac{\pi}{3}} \right\}.
\end{equation*}
We define a decomposition of the boundary $\partial{\Delta}=\partial^-\Delta\cup \partial^0\Delta\cup \partial^+\Delta$, where
\begin{align*}
\partial^\pm \Delta &:= \left\{ re^{\sqrt{-1}\theta} \in \Delta \mid {0\leq r \leq1}, \ {\theta=\pm\frac{\pi}{3}} \right\}, \\
\partial^0\Delta &:= \left\{ e^{\sqrt{-1}\theta} \in \Delta \mid {-\frac{\pi}{3} \leq \theta \leq \frac{\pi}{3}} \right\}.
\end{align*}
Set $P:=\Sigma_{p,b}$ and $U_{p,b} := P\times\Delta$.
We see $U_{p,b} \cong \natural^{2p+b-1}(S^1\times D^3)$.
Then, decompose the boundary of $U_{p,b}$ as $\partial{U_{p,b}}=\partial^-{U_{p,b}} \cup \partial^0U_{p,b} \cup \partial^+{U_{p,b}}$, where
\begin{align*}
\partial^\pm {U_{p,b}} := P\times\partial^\pm{\Delta}
\quad \textrm{and} \quad
\partial^0{{U_{p,b}}} := (P\times\partial^0{\Delta})\cup(\partial{P}\times\Delta).
\end{align*}

Next, we set $V_{n}:= \natural^{n}(S^1\times D^3)$, where $n=k-2p-b+1$.
Then, let $\partial{V_n} = \partial^-{V_{g,k;p,b}}\cup\partial^+{V_{g,k;p,b}}$ be a genus-$(g-p)$ Heegaard splitting of $\#^{n}(S^1\times S^2)$.

Finally, we define $Z_{k} := {U_{p,b}}\bcs {V_{n}}$, where the boundary connected sum is taken in neighborhoods of points in $\Int(\partial^-{U_{p,b}}\cap \partial^+{U_{p,b}})$ and $\partial^-{V_{g,k;p,b}}\cap \partial^+{V_{g,k;p,b}}$ (see Figure~\ref{fig:connsum-UV}).
\begin{figure}[!tbp]
\centering
\includegraphics[scale=0.7]{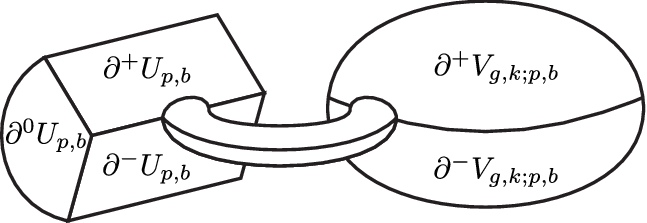}
\caption{The boundary connected sum  $Z_{k}={U_{p,b}} \bcs {V_{n}}$.}
\label{fig:connsum-UV}
\end{figure}
Set $Y_{k} :=\partial{Z_{k}}$, and give a decomposition
\begin{align*}
Y_{k} = Y^-_{g,k;p,b}\cup Y^0_{p,b}\cup Y^+_{g,k;p,b},
\end{align*}
where
\begin{align*}
 Y^\pm_{g,k;p,b} := \partial^{\pm}{U_{p,b}} \bcs \partial^{\pm}{{V_{g,k;p,b}}} \quad \textrm{and} \quad Y^0_{p,b} := \partial^0{U_{p,b}}.
\end{align*}
Then the following properties hold (see Section~3 and 4 of \cite{CasGayPin18a}).
\begin{itemize}
\item
The $4$-manifold $Z_{k}$ is diffeomorphic to $\natural^{k}(S^1\times D^3)$.
Thus, we have $Y_{k}=\partial{Z_{k}}\cong\#^k(S^1\times S^2)$.
\item
The $3$-manifolds $Y^\pm_{g,k;p,b}$ are obtained from $\Sigma_{g,b}\times[0,1]$ by attaching $3$-dimensional $2$-handles along $g-p$ disjoint simple closed curves on $\Sigma_{g,b}\times\{1\}$ to get $\Sigma_{p,b}$. That is, $Y^\pm_{g,k;p,b}$ are diffeomorphic to compression bodies from $\Sigma_{g,b}$ to $\Sigma_{p,b}$.
\item The $3$-manifold $Y^0_{p,b}$ can be identified with $P\times[-1,1]$.
\item The decomposition $Y^-_{g,k;p,b}\cup Y^+_{g,k;p,b}$ is a sutured Heegaard splitting of the $3$-manifold
\begin{equation*}
(P\times[-1,1])\#(\#^{n}S^1\times S^2),
\end{equation*}
where $n=k-2p-b+1$. The splitting surface $Y^-_{g,k;p,b}\cap Y^+_{g,k;p,b}$ is diffeomorphic to $\Sigma_{g,b}$.
\end{itemize}

We now proceed to the definition of a relative trisection.

\begin{definition}[\cite{CasGayPin18a}]\label{dfn:rt}
Let $W$ be a compact, connected, oriented, smooth $4$-manifold with connected boundary.
Let $g$, $k$, $p$, and $b$ be integers such that $g,k,p\geq0$, $b\geq1$, and $2p+b-1\leq k\leq g+p+b-1$.
A decomposition $\calRT:W=W_1\cup W_2\cup W_3$ is called a $(g,k;p,b)$-\textit{relative trisection} of $W$ if it satisfies the following conditions.
\begin{itemize}
\item
For each $i \in\{1,2,3\}$, there exists a diffeomorphism $\varphi_i:W_i\to Z_{k}$.
\item
For each $i \in\{1,2,3\}$, taking indices mod $3$,
\begin{equation*}
\varphi_i(W_i\cap W_{i\pm1})=Y^\mp_{g,k;p,b} \quad\textrm{and}\quad \varphi_{i}(W_i\cap \partial{W})=Y^0_{p,b}.
\end{equation*}
\end{itemize}
\end{definition}

By Definition~\ref{dfn:rt}, if a decomposition $W=W_1\cup W_2\cup W_3$ is a $(g,k;p,b)$-relative trisection, then the following properties hold.
\begin{itemize}
\item
For each $i\in\{1,2,3\}$, the sector $W_i$ is diffeomorphic to the genus-$k$ $4$-dimensional $1$-handlebody $\natural^k(S^1\times D^3)$.
\item
For each $i,j\in\{1,2,3\}$ such that $i\neq j$, the double intersection $W_i\cap W_j(=\partial{W_i}\cap\partial{W_j})$ is diffeomorphic to the $3$-manifold obtained from $\Sigma_{g,b}\times[0,1]$ by attaching $3$-dimensional $2$-handles along $g-p$ disjoint simple closed curves on $\Sigma_{g,b}\times\{1\}$ to get $\Sigma_{p,b}$.
That is, $W_i\cap W_j$ is diffeomorphic to compression bodies from $\Sigma_{g,b}$ to $\Sigma_{p,b}$. 
\item
The triple intersection $W_1\cap W_2\cap W_3$ is diffeomorphic to the surface $\Sigma_{g,b}$ of genus $g$ with $b$ boundary components.
\end{itemize}

\begin{theorem}[{\cite{GayKir16}}]\label{thm:rt-prop}
%
The following statements hold for compact, connected, oriented, smooth $4$-manifolds with connected boundary.
\begin{enumerate}
\item Any such $4$-manifold with boundary admits a relative trisection.
\item A $(g,k;p,b)$-relative trisection $\calRT$ of a $4$-manifold $W$ with boundary induces an open book on the boundary $\partial{W}$ with pages $\Sigma_{p,b}$.
\item For any $4$-manifold $W$ with boundary and open book $\calO$ on $\partial{W}$, there exists a relative trisection $\calRT$ of $W$ that induces $\calO$.
%
%
\end{enumerate}
\end{theorem}

Henceforth, any $4$-manifold with boundary is assumed to have nonempty connected boundary.

\subsection{Trisection diagram}\label{sub:trisdiag}

A trisection is encoded by a diagram $(\Sigma; \alpha, \beta, \gamma)$, where $\Sigma$ is a surface and $\alpha$, $\beta$, $\gamma$ are collections of simple closed curves.
In this subsection, we recall the definitions of (relative) trisection diagrams, following Castro, Gay, and Pinz\'{o}n-Caicedo~\cite{CasGayPin18a}.
First, we define several equivalence relations for tuples of surfaces and collections of curves.

\begin{definition}
Let $\Sigma$ and $\Sigma'$ be compact, connected, oriented surfaces. For $i\in \{1,\ldots,n\}$, let $\eta^i=\{\eta^i_1,\ldots,\eta^i_k\}$ and $\zeta^i=\{\zeta^i_1,\ldots,\zeta^i_k\}$ be collections of $k$ pairwise disjoint simple closed curves on $\Sigma$ and $\Sigma'$, respectively.
Then, we consider two $(n+1)$-tuples $\calD:=(\Sigma;\eta^1,\ldots,\eta^n)$ and $\calD':=(\Sigma';\zeta^1,\ldots,\zeta^n)$.
\begin{itemize}
\item
Suppose $\Sigma=\Sigma'$. The two diagrams $\calD$ and $\calD'$ are \textit{isotopic} if for any $i\in\{1,\ldots,n\}$ there exists an ambient isotopy $\{\varphi_t^i:\Sigma\to\Sigma\}_{t\in[0,1]}$ such that $\varphi_1^i(\eta^i)=\zeta^i$.
\item
Suppose $\Sigma=\Sigma'$. The two diagrams $\calD$ and $\calD'$ are \textit{slide-equivalent} if for each $i\in\{1,\ldots,n\}$, $\eta^i$ and $\zeta^i$ are related by a sequence of handleslides and ambient isotopies.
\item
The diagrams $\calD$ and $\calD'$ are \textit{diffeomorphism and handleslide equivalent} if there exists a diffeomorphism $f:\Sigma\to \Sigma'$ such that $f(\calD)=(f(\Sigma);f(\eta^1),\ldots,f(\eta^n))$ and $\calD'$ are slide-equivalent.
\end{itemize}
\end{definition}

\begin{definition}\label{def:td}
Let $g$ and $k$ be integers such that $0\leq k \leq g$.
Let $\Sigma$ be a surface diffeomorphic to $\Sigma_{g}$, and let $\alpha=\{\alpha_1,\ldots,\alpha_g\}$, $\beta=\{\beta_1,\ldots,\beta_g\}$, and $\gamma=\{\gamma_1,\ldots,\gamma_g\}$ be collections of $g$ pairwise disjoint simple closed curves on $\Sigma$.
The $4$-tuple $\calD=(\Sigma;\alpha,\beta,\gamma)$ is called a $(g,k)$-\textit{trisection diagram} if $(\Sigma;\alpha,\beta)$, $(\Sigma;\beta,\gamma)$, and $(\Sigma;\gamma,\alpha)$ are diffeomorphism and handleslide equivalent to the standard diagram $(\Sigma_{g};\delta,\epsilon)$ shown in Figure \ref{fig:td-stdgk}, where the red curves denote $\delta=\{\delta_1,\ldots,\delta_g\}$ and the blue curves denote $\epsilon=\{\epsilon_1,\ldots,\epsilon_g\}$.
\end{definition}
\begin{figure}[!htbp]
\centering
\includegraphics[scale=0.7]{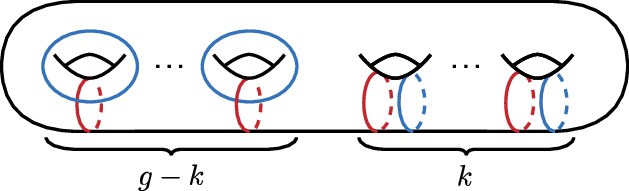}
\caption{The standard diagram $(\Sigma_{g,b};\delta,\epsilon)$.}
\label{fig:td-stdgk}
\end{figure}

Note that the diagram $(\Sigma_{g};\delta,\epsilon)$ in Figure~\ref{fig:td-stdgk} is the standard genus-$g$ Heegaard diagram of $\#^k(S^1\times S^2)$.

\begin{definition}\label{def:rtd}
Let $g$, $k$, $p$, and $b$ be integers such that $g,k,p\geq0$, $b\geq1$, and $2p+b-1\leq k\leq g+p+b-1$.
Let $\Sigma$ be a surface diffeomorphic to $\Sigma_{g,b}$, and let $\alpha=\{\alpha_1,\ldots,\alpha_{g-p}\}$, $\beta=\{\beta_1,\ldots,\beta_{g-p}\}$, and $\gamma=\{\gamma_1,\ldots,\gamma_{g-p}\}$ be collections of $g-p$ pairwise disjoint simple closed curves on $\Sigma$.
The $4$-tuple $\calD=(\Sigma;\alpha,\beta,\gamma)$ is called a $(g,k;p,b)$-\textit{relative trisection diagram} if $(\Sigma;\alpha,\beta)$, $(\Sigma;\beta,\gamma)$, and $(\Sigma;\gamma,\alpha)$ are diffeomorphism and handleslide equivalent to the standard diagram $(\Sigma_{g,b};\delta,\epsilon)$ shown in Figure \ref{fig:td-stdgkpb}, where the red curves denote $\delta=\{\delta_1,\ldots,\delta_{g-p}\}$ and the blue curves denote $\epsilon=\{\epsilon_1,\ldots,\epsilon_{g-p}\}$.
\end{definition}
\begin{figure}[!htbp]
\centering
\includegraphics[scale=0.7]{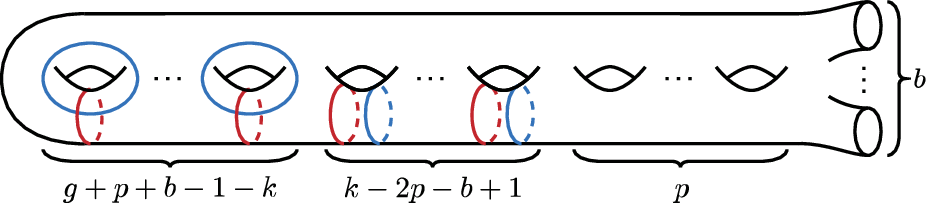}
\caption{The standard diagram $(\Sigma_{g,b};\delta,\epsilon)$.}
\label{fig:td-stdgkpb}
\end{figure}


For a (relative) trisection diagram $(\Sigma;\alpha,\beta,\gamma)$, we represent the $\alpha$, $\beta$, and $\gamma$ curves in red, blue, and green, respectively. See Figure~\ref{fig:td-gd} for an example.


There is a correspondence between trisections and trisection diagrams.

\begin{theorem}[{\cite{CasGayPin18a}}]\label{thm:rt-rtd}
The following statements hold.
\begin{enumerate}
\item
For any $(g,k)$- (or $(g,k;p,b)$-relative) trisection diagram $(\Sigma;\alpha,\beta,\gamma)$, there exists a unique (up to diffeomorphism) trisection $X=X_1\cup X_2\cup X_3$ satisfying the following conditions.
\begin{itemize}
\item
$X_1\cap X_2\cap X_3\cong \Sigma$.
\item
Under this identification, the $\alpha$-, $\beta$-, and $\gamma$-curves bound compressing disks in $X_3 \cap X_1$, $X_1 \cap X_2$, and $X_2 \cap X_3$, respectively.
\end{itemize}
\item
For any (relative) trisection $\calT$, there exists a (relative) trisection diagram $\calD$ such that $\calT$ is induced from $\calD$ by (i).
\item
Let $\calD$ and $\calD'$ be (relative) trisection diagrams.
If the induced (relative) trisections $\calT_{\calD}$ and $\calT_{\calD'}$ are diffeomorphic, then $\calD$ and $\calD'$ are diffeomorphism and handleslide equivalent.
\end{enumerate}
\end{theorem}

The construction in (ii) proceeds as follows.
Given a trisection $X=X_1\cup X_2\cup X_3$, we set the central surface as $\Sigma:=X_1\cap X_2\cap X_3$. 
We then choose collections of simple closed curves $\alpha$, $\beta$, and $\gamma$ on $\Sigma$ forming complete disk systems of the handlebodies $X_3 \cap X_1$, $X_1 \cap X_2$, and $X_2 \cap X_3$, respectively.
The resulting $4$-tuple $(\Sigma;\alpha,\beta,\gamma)$ is then a trisection diagram.

Regarding (i), the desired trisected $4$-manifold is obtained by thickening the trisection diagram via taking the product with $D^{2}$, attaching $2$-handles along the curves $\alpha$, $\beta$, $\gamma$ as indicated, and then capping off with $1$-handlebodies.
The similar construction applies, with minor modifications, to the case of relative trisections.
For details, see \cite{CasGayPin18a}. 
For a trisection diagram $\calD$, we denote by $\calT(\calD)$ (resp. $\calRT(\calD)$ in the relative case) the trisection constructed in this way, and by $X(\calD)$ the underlying $4$-manifold.

Theorem~\ref{thm:rt-rtd} establishes the following one-to-one correspondence:
\begin{equation*}
\frac{\{\text{(relative) trisections}\}}{\text{diffeomorphism}}
\ \ {\stackrel{1:1}{\longleftrightarrow}} \ \ 
\frac{\{\text{(relative) trisection diagrams}\}}{\text{diffeomorphism and handleslide equivalent}}.
\end{equation*}

\section{Distinguishing relative trisections by the capping operation}

In \cite{Tak24}, the author proved the following result by using Theorem~\ref{thm:equiv-rtd-gd}.

\begin{theorem}\label{thm:nondiffeo-tak24}
The two $(2,1;0,2)$-relative trisection diagrams shown in Figure~\ref{fig:td-S2xD2} both represent $S^2\times D^2$.
These diagrams are not diffeomorphism and handleslide equivalent.
Consequently, the two $(2,1;0,2)$-relative trisections obtained by these diagrams are not diffeomorphic.
In addition, these relative trisections are of minimal genus and induce equivalent open books.
\end{theorem}
\begin{figure}[!htbp]
\centering
\includegraphics[scale=0.7]{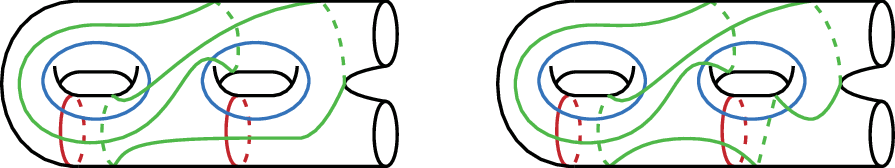}
\caption{$(2,1;0,2)$-relative trisection diagrams of $S^2\times D^2$.}
\label{fig:td-S2xD2}
\end{figure}

The closed trisection diagrams obtained by capping off the left and right diagrams in Figure~\ref{fig:td-S2xD2} represent $S^2 \times S^2$ and $S^2 \tilde{\times} S^2$, respectively.
Since these closed $4$-manifolds are not homeomorphic, it follows from Theorem~\ref{thm:equiv-rtd-gd} that the original relative trisection diagrams are not diffeomorphism and handleslide equivalent.
Using this method, we present additional examples of non-diffeomorphic relative trisections of the same $4$-manifold with boundary.

\subsection{Proofs of Theorems~\ref{thm:cl+1hdlbd} and \ref{thm:-D4}}

Starting from non-diffeomorphic trisections of a closed $4$-manifold, we describe a method for producing non-diffeomorphic relative trisections of its connected sum with a $1$-handlebody.

\begin{proof}[Proof of Theorem~\ref{thm:cl+1hdlbd}]
Let $\calT:X=X_1 \cup X_2 \cup X_3$ and $\calT':X=X'_1 \cup X'_2 \cup X'_3$ be non-diffeomorphic $(g,k)$-trisections of the closed $4$-manifold $X$.
Let $\calD = (\Sigma; \alpha, \beta, \gamma)$ and $\calD' = (\Sigma'; \alpha', \beta', \gamma')$ be $(g,k)$-trisection diagrams of $\calT$ and $\calT'$, respectively.
Here, $\Sigma := X_1 \cap X_2 \cap X_3$, and $\alpha$, $\beta$, $\gamma$ are meridian systems of the 3-dimensional 1-handlebodies $X_3 \cap X_1$, $X_1 \cap X_2$, $X_2 \cap X_3$, respectively.
The diagram $\calD'$ is defined similarly.

Let $\calT_{b}$ be the $(0,b-1;0,b)$-relative trisection of $\natural^{b-1}(S^1\times D^3)$ given by the diagram $(\Sigma_{0,b}; \emptyset, \emptyset, \emptyset)$.
It induces the open book $(\Sigma_{0,b}, \mathrm{id})$ on the boundary $\#^{b-1}(S^1\times S^2)$.
We consider the $(g,k+b-1;0,b)$-relative trisections $\calT\#\calT_{b}$ and $\calT'\#\calT_{b}$ of the $4$-manifold $X\#(\natural^{b-1}(S^1\times D^3))$.
(For details of the connected sum operation between a trisection of a closed $4$-manifold and a relative trisection of a $4$-manifold with boundary, see Appendix~A.)
These relative trisections also induce the open book $(\Sigma_{0,b}, \text{id})$.

In the following, we prove that $\calT \# \calT_{b}$ and $\calT' \# \calT_{b}$ are not diffeomorphic as trisections.
We consider the following connected sums of the diagrams:
\begin{align*}
\calD \# (\Sigma_{0,b}; \emptyset, \emptyset, \emptyset) &= (\Sigma \# \Sigma_{0,b}; \alpha, \beta, \gamma)
\quad \text{and} \\
\calD' \# (\Sigma_{0,b}; \emptyset, \emptyset, \emptyset) &= (\Sigma' \# \Sigma_{0,b}; \alpha', \beta', \gamma').
\end{align*}
Figure~\ref{fig:td-bcs-1hdbd} shows an example of a connected sum $(\Sigma;\alpha,\beta,\gamma)\#(\Sigma_{0,b}; \emptyset, \emptyset, \emptyset)$.
Here, the $(g, k+b-1 ; 0 ,b)$-relative trisection diagrams $\calD \# (\Sigma_{0,b}; \emptyset, \emptyset, \emptyset)$ and $\calD' \# (\Sigma_{0,b}; \emptyset, \emptyset, \emptyset)$ represent $\calT \# \calT_{b}$ and $\calT' \# \calT_{b}$, respectively.
\begin{figure}[!tbp]
\centering
\includegraphics[scale=0.7]{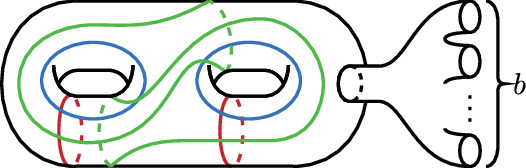}
\caption{An example of a connected sum $(\Sigma;\alpha,\beta,\gamma)\#(\Sigma_{0,b}; \emptyset, \emptyset, \emptyset)$.}
\label{fig:td-bcs-1hdbd}
\end{figure}
Then we see that
\begin{align*}
\rtdcap{\calD \# (\Sigma_{0,b}; \emptyset, \emptyset, \emptyset)})
&= (\rtdcap{\Sigma \# \Sigma_{0,b}}; \alpha, \beta, \gamma) \\
&= (\Sigma \# S^2; \alpha, \beta, \gamma) \\
&= (\Sigma; \alpha, \beta, \gamma) \\
&= \calD.
\end{align*}
Similarly, we have $\rtdcap{\calD' \# (\Sigma_{0,b}; \emptyset, \emptyset, \emptyset)} = \calD'$.

We argue by contradiction.
Assume that $\calT \# \calT_{b}$ and $\calT' \# \calT_{b}$ are diffeomorphic.
Then, by Theorem~\ref{thm:rt-rtd}, we see that $\calD \# (\Sigma_{0,b}; \emptyset, \emptyset, \emptyset)$ and $\calD' \# (\Sigma_{0,b}; \emptyset, \emptyset, \emptyset)$ are diffeomorphism and handleslide equivalent.
Hence, by Theorem~\ref{thm:equiv-rtd-gd}, we conclude that $\calD$ and $\calD'$ are diffeomorphism and handleslide equivalent.
Therefore, $\calT$ and $\calT'$ are diffeomorphic, which is a contradiction.
\end{proof}

From the case $n=1$ of Theorem~\ref{thm:cl+1hdlbd}, we deduce that if a closed $4$-manifold $X$ admits non-diffeomorphic $(g,k)$-trisections, then $X-\Int{D^4} (\cong X\#D^4)$ also admits non-diffeomorphic $(g,k;0,1)$-relative trisections.
Theorem~\ref{thm:-D4} is obtained by combining with Proposition~\ref{thm:tg-eq}.

\begin{proof}[Proof of Proposition~\ref{thm:tg-eq}]
First, we check that $g(X-\Int{D^4})\leq g(X)$.
This follows from the fact that a genus-$g(X)$ relative trisection of $X\#{D^4}$ can be obtained by taking the connected sum of a minimal genus trisection of $X$ with the $(0,0;0,1)$-relative trisection of $D^4$.

Next, we prove that $g(X)\leq g(X-\Int{D^4})$.
Set $g:=g(X-\Int{D^4})$.
We want to show that $g(X)\leq g$, which means $X$ admits a genus-$g$ closed trisection.
Let $\calRT$ be a $(g,k;p,b)$-relative trisection of $X-\Int{D^4}$ that realizes the minimal genus.
By Lemma~\ref{lem:rt-pb=01} described later, we have $p=0$ and $b=1$.
Hence, $\calRT$ induces the open book $(D^2,\id)$ on $S^3$.
Recall that the $(0,0;0,1)$-relative trisection of $D^4$ induces the same open book.
By \cite[Proposition~2.12]{CasOzb19}, the two relative trisections can be glued together.
As a result, the $4$-manifold $(X-\Int{D^4})\cup_{S^3}D^4 \cong X$ admits a $(g,k)$-trisection.
\end{proof}

\begin{lemma}\label{lem:rt-pb=01}
Let $W$ be a $4$-manifold with boundary $\partial{W}\cong S^3$.
Let $\calRT$ be a $(g,k;p,b)$-relative trisection of $W$.
If $g$ is equal to the trisection genus of $W$, then $p=0$ and $b=1$.
\end{lemma}

\begin{proof}
Let $\calO(\calRT):=(\Sigma_{p,b},\phi)$ be the open book induced by $\calRT$.
Since both $\calO(\calRT)$ and $(D^2,\id)$ are open books of $S^3$, they are related by a sequence of Hopf stabilizations and destabilizations (see \cite[Theorem~1]{GirGoo06}).
%
Then there exists a $(g',k';0,1)$-relative trisection $\calRT'$ of $W$, with induced open book $\calO(\calRT')$ equivalent to $(D^2,\id)$, such that $\calRT$ and $\calRT'$ are related by a sequence of relative stabilizations and destabilizations (see the proof of Theorem~8 in \cite{Cas16}).
According to \cite[p.32]{Cas16}, there are two types of relative stabilizations.
We refer to them as Type~I and Type~II relative stabilizations (see Figure~\ref{fig:relstab}).
\begin{figure}[!tbp]
\centering
\includegraphics[scale=0.7]{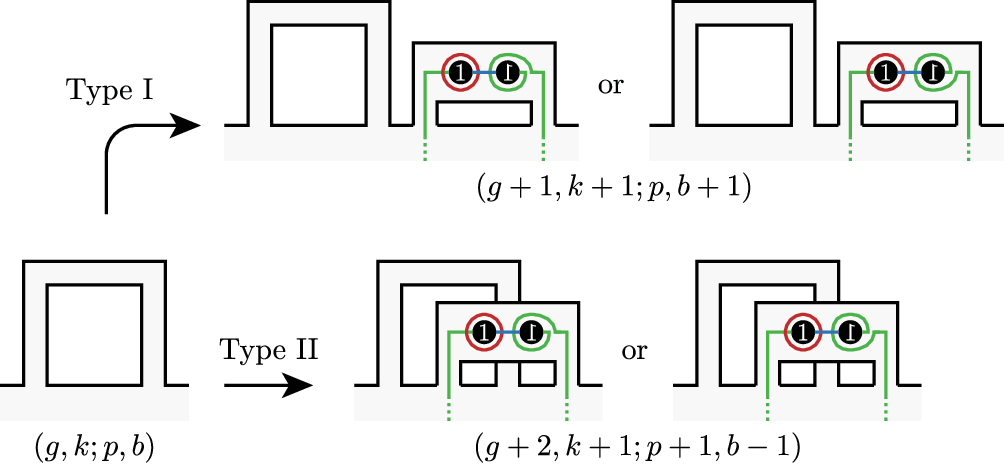}
\caption{Relative stabilizations.}
\label{fig:relstab}
\end{figure}
In a Type~I relative stabilization, the parameters $(g,k;p,b)$ change to $(g+1,k+1;p,b+1)$.
In a Type~II relative stabilization, the parameters $(g,k;p,b)$ change to $(g+2,k+1;p+1,b-1)$.

For the sequence of relative (de)stabilizations relating $\calRT$ and $\calRT'$, let us denote:
\begin{itemize}
    \item $l^+$: the number of Type~I relative stabilizations,
    \item $l^-$: the number of Type~I relative destabilizations,
    \item $m^+$: the number of Type~II relative stabilizations,
    \item $m^-$: the number of Type~II relative destabilizations.
\end{itemize}
Since the parameters of $\calRT$ are $(g,k;p,b)$ and those of $\calRT'$ are $(g',k';0,1)$, their differences are given by
\begin{align*}
g'-g &= (l^{+} + 2m^{+}) - (l^{-} + 2m^{-}), \\
k'-k &= (l^{+} + m^{+}) - (l^{-} + m^{-}), \\
0-p &= m^{+} - m^{-}, \\
1-b &= l^{+} - l^{-} + m^{-} - m^{+}.
\end{align*}
The difference in the genera of the relative trisections $\calRT$ and $\calRT'$ is computed as follows:
\begin{align*}
g'-g &= (l^{+} + 2m^{+}) - (l^{-} + 2m^{-}) \\
&= (l^{+} - l^{-} + m^{-} - m^{+}) + 3(m^{+} - m^{-}) \\
&= (1-b) + 3(0-p) \\
&= 1-b-3p.
\end{align*}
Since $g$ is the trisection genus of $W$, we have $g'-g\geq0$.
This implies that $1-b-3p \geq 0$.
Since $p \geq 0$ and $b \geq 1$ are required by the definition of a relative trisection, the only possible case is $p=0$ and $b=1$.
\end{proof}

\subsection{Proof of Theorem~\ref{thm:ndrt-inf}}\label{subsec:infinite}

In this subsection, we give examples of non-diffeomorphic relative trisections other than those discussed in the previous subsection.
That is, the underlying $4$-manifolds cannot be realized as the connected sum of a closed $4$-manifold and a $1$-handlebody.
We prove that the boundary connected sum of $n$ copies of $S^2 \times D^2$ satisfies the conditions of Theorem~\ref{thm:ndrt-inf}.

\begin{proof}[Proof of Theorem~\ref{thm:ndrt-inf}]
Let $\calD$ and $\calD'$ be the $(2,1;0,2)$-relative trisection diagrams of $S^2\times D^2$ shown on the left and right sides in Figure~\ref{fig:td-S2xD2}, respectively.
For a positive integer $n$, let $\natural^n\calD$ be the $(2n,n;0,n+1)$-relative trisection diagram obtained by taking the boundary connected sum of $n$ copies of $\calD$.
The diagram $\natural^n\calD'$ is defined similarly.
Both of these diagrams represent the $4$-manifold $\natural^{n}(S^2\times D^2)$.
In addition, since both $\calD$ and $\calD'$ induce the same open book $(\Sigma_{0,2},\id)$, the diagrams $\natural^n\calD$ and $\natural^n\calD'$ also induce the same open book $(\Sigma_{0,n+1},\id)$.

We now prove that the two relative trisection diagrams $\natural^n\calD$ and $\natural^n\calD'$ are not diffeomorphism and handleslide equivalent.
By considering the capped diagrams, we have:
\begin{align*}
X(\rtdcap{\natural^n\calD} \cong \#^n(S^2\times S^2)
\quad\text{and}\quad
X(\rtdcap{\natural^n\calD'} \cong \#^n(\CC{P^2}\#\CC{P^2}).
\end{align*}
%
Since these closed $4$-manifolds are not homeomorphic, it follows from Theorem~\ref{thm:equiv-rtd-gd} that $\natural^n\calD$ and $\natural^n\calD'$ are not diffeomorphism and handleslide equivalent.
Consequently, the corresponding $(2n,n;0,n+1)$-relative trisections $\calRT(\natural^n\calD)$ and $\calRT(\natural^n\calD')$ are not diffeomorphic.


%
Finally, we prove that $g(\natural^n(S^2 \times D^2))\geq2n$. 
%
%
For the sake of contradiction, assume that $\natural^n(S^2 \times D^2)$ admits a $(2n-1, k; p, b)$-relative trisection for some integers $k, p\geq0$ and $b\geq1$.
By Proposition~\ref{thm:rt-prop}, the boundary $\#^n(S^2 \times S^1)$ has an open book decomposition with page $\Sigma_{p,b}$.
Thus $\#^n(S^2 \times S^1)$ admits a Heegaard splitting of genus $2p+b-1$.
Since the Heegaard genus of $\#^n(S^2 \times S^1)$ is known to be $n$, we have the inequality:
\begin{equation}\label{eq:lower_n}
n \leq 2p + b - 1.
\end{equation}

On the other hand, by Corollary~2.10 in \cite{CasOzb19}, we obtain
\begin{equation*}
(2n - 1) - 3k + 3p + 2b - 1 = \chi(\natural^n(S^2 \times D^2)) = n+1, \end{equation*}
which simplifies to $k = p + b - 1 + {(n - b)}/{3}$.
Since $k$ must be an integer, there exists an integer $m$ such that $b=n-3m$.
Substituting this into the expression for $k$, we obtain $k = p - 1 + n-2m$.

By Definition~\ref{dfn:tris-model}, we have $2p+b-1\leq k$.
Substituting $b=n-3m$ and $k=p-1+n-2m$ into this inequality yields
\begin{equation*}
2p + (n - 3m) - 1 \leq p - 1 + n - 2m,
\end{equation*}
which simplifies to $p \leq m$.
Under this condition, we compute
\begin{align*}
2p+b-1 &= 2p + (n - 3m) - 1 \\
&\leq 2m +(n - 3m) - 1 \quad (\text{since } 0\leq p\leq m) \\
&= n-m-1 \\
&<n.
\end{align*}
This contradicts the inequality \eqref{eq:lower_n}. 
\end{proof}

\subsection{Proof of Theorem~\ref{thm:ndrt-Ac}}

In this subsection, we present non-diffeomorphic minimal genus relative trisections of the Akbulut cork $W^{-}(0,0)$, which is a contractible $4$-manifold given by the handlebody diagram in Figure~\ref{fig:kd-Ac}.
\begin{figure}[!tbp]
\centering
\includegraphics[scale=0.7]{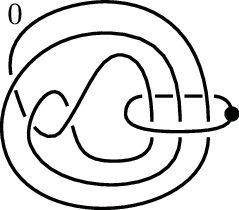}
\caption{A handlebody diagram of the Akbulut cork $W^{-}(0,0)$.}
\label{fig:kd-Ac}
\end{figure}

\begin{theorem}[\cite{Tak25b}]
Let $\calD_1$ and $\calD_2$ be the diagrams shown in Figure~\ref{fig:td-Ac}.
Then, they are $(3,3;0,4)$-relative trisection diagrams of the Akbulut cork.
Furthermore, the trisection genus of the Akbulut cork is $3$.
\end{theorem}
\begin{figure}[!tbp]
\centering
\includegraphics[scale=0.7]{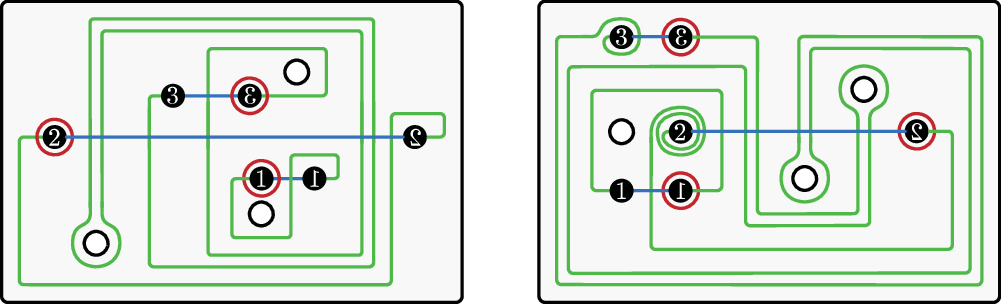}
\caption{Two $(3,3;0,4)$-relative trisection diagrams of the Akbulut cork. Left: $\calD_1$. Right: $\calD_2$.}
\label{fig:td-Ac}
\end{figure}

The next statement implies Theorem~\ref{thm:ndrt-Ac}. 

\begin{proposition}
The two $(3;3,0,4)$-relative trisection diagrams $\calD_1$ and $\calD_2$ of the Akbulut cork are not diffeomorphism and handleslide equivalent.
\end{proposition}

\begin{proof}
The capped diagrams $\rtdcap{\calD_1}$ and $\rtdcap{\calD_2}$ can be modified into the diagrams shown in Figure~\ref{fig:td-Ac-cap} by surface diffeomorphisms and isotopies of the curves.
\begin{figure}[!tbp]
\centering
\includegraphics[scale=0.7]{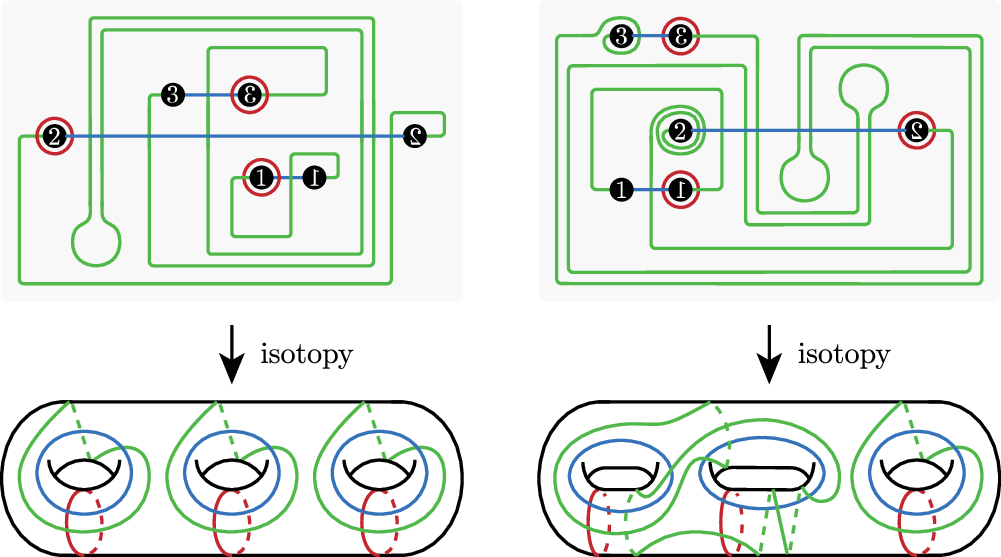}
\caption{Modifications of $\rtdcap{\calD_1}$ and $\rtdcap{\calD_2}$.}
\label{fig:td-Ac-cap}
\end{figure}
The left diagram $\rtdcap{\calD_1}$ is the standard trisection diagram of $3\overline{\CC{P^2}}$.
%
By Lemma~13 in \cite{GayKir16}, we see that the right diagram $\rtdcap{\calD_2}$ represents the $4$-manifold given by the handlebody diagram in Figure~\ref{fig:Kd-S2xS2-CP2}.
\begin{figure}[!tbp]
\centering
\includegraphics[scale=0.7]{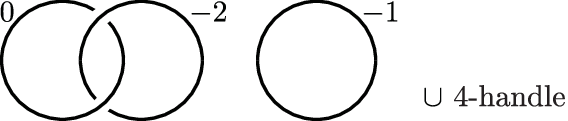}
\caption{The handlebody diagram induced from $\rtdcap{\calD_2}$.}
\label{fig:Kd-S2xS2-CP2}
\end{figure}
This $4$-manifold is diffeomorphic to $\CC{P^2}\#2\overline{\CC{P^2}}$.
%
Since $X(\rtdcap{\calD_1}) \not\cong X(\rtdcap{\calD_2})$, it follows from Theorem~\ref{thm:equiv-rtd-gd} that $\calD_1$ and $\calD_2$ are not diffeomorphism and handleslide equivalent.
\end{proof}

\begin{remark}
Another possible approach to prove that $\calD_{1}$ and $\calD_{2}$ are not equivalent is to compare the induced open books.
An algorithm for computing the monodromy of the open book induced by a relative trisection diagram is given in \cite[Theorem~5]{CasGayPin18a}.
However, the procedure requires performing very complicated modifications.
At present, we have not yet determined the monodromy of the open books induced by $\calD_{1}$ and $\calD_{2}$.
\end{remark}

\subsection{Proof of Theorem~\ref{thm:non-min-genus}}

In this subsection, we prove that any $4$-manifold with boundary admits a pair of non-diffeomorphic relative trisections.
As a preparation for the proof, we first consider two relative trisections of the $4$-ball $D^4$.
Let $\calD^{+}$ and $\calD^{-}$ be the $(1,1;0,2)$-relative trisection diagrams of $D^4$ shown in Figures~\ref{fig:td-B4+} and \ref{fig:td-B4-}, respectively.
\begin{figure}[!tbp]
  \begin{minipage}[b]{0.49\linewidth}
    \centering
    \includegraphics[scale=0.7]{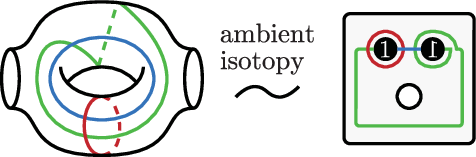}
    \caption{$\calD^{+}$}
    \label{fig:td-B4+}
  \end{minipage}
  \begin{minipage}[b]{0.49\linewidth}
    \centering
    \includegraphics[scale=0.7]{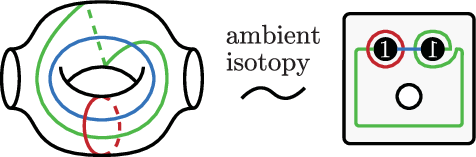}
    \caption{$\calD^{-}$}
    \label{fig:td-B4-}
  \end{minipage}
\end{figure}

\begin{lemma}
The $(1,1;0,2)$-relative trisection diagrams $\calD^{+}$ and $\calD^{-}$ are not diffeomorphism and handleslide equivalent.
Thus, the induced $(1,1;0,2)$-relative trisections $\calRT(\calD^+)$ and $\calRT(\calD^-)$ of $D^4$ are not diffeomorphic.
\end{lemma}

\begin{proof}
We easily see that $\rtdcap{\calD^{\pm}}$ are $(1,0)$-trisection diagrams of $\pm\CC{P^2}$.
By Theorem~\ref{thm:equiv-rtd-gd}, it follows that $\calD^{+}$ and $\calD^{-}$ are not diffeomorphism and handleslide equivalent.

Alternatively, one may compare the open books induced by the respective relative trisection diagrams.
In both cases, the page is an annulus $\Sigma_{0,2}$.
The monodromy induced by $\calD^{+}$ is a right-handed Dehn twist, whereas that induced by $\calD^{-}$ is a left-handed Dehn twist.
\end{proof}

For a relative trisection diagram $\calD$, by taking the boundary connected sum with $\calD^{+}$ and $\calD^{-}$, we can obtain non-diffeomorphic relative trisections.

\begin{proof}[Peoof of Theorem~\ref{thm:non-min-genus}]
It is known that any closed $3$-manifold admits a planar open book decomposition (see \cite[Theorem~10K1]{Rol76}).
%
Thus, the boundary $\partial{W}$ admits a planar open book $(\Sigma_{0,b}, \phi)$.
By (iii) of Theorem~\ref{thm:rt-prop}, there exists a $(g,k;0,b)$-relative trisection diagram $\calD=(\Sigma;\alpha,\beta,\gamma)$ of $W$ that induces the planar open book $(\Sigma_{0,b}, \phi)$.

We consider the boundary connected sums $\calD \bcs \calD^{\pm}$ along an arc $c\subset\partial{\Sigma}$ (see Figures~\ref{fig:td-glue-D4+} and \ref{fig:td-glue-D4-}).
\begin{figure}[!tbp]
  \begin{minipage}[b]{0.49\linewidth}
    \centering
    \includegraphics[scale=0.7]{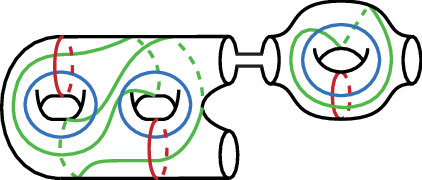}
    \caption{$\calD \bcs \calD^{+}$}
    \label{fig:td-glue-D4+}
  \end{minipage}
  \begin{minipage}[b]{0.49\linewidth}
    \centering
    \includegraphics[scale=0.7]{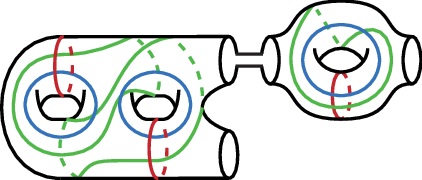}
    \caption{$\calD \bcs \calD^{-}$}
    \label{fig:td-glue-D4-}
  \end{minipage}
\end{figure}
The resulting diagrams $\calD \bcs \calD^{\pm}$ are $(g+1,k;0,b+1)$-relative trisection diagrams of $W\bcs D^4\cong W$.
We now prove that the induced relative trisections $\calRT(\calD \bcs \calD^{+})$ and $\calRT(\calD \bcs \calD^{-})$ are not diffeomorphic.
It can be observed from Figure~\ref{fig:td+cp2-gd} that
\begin{equation*}
\rtdcap{ \calD \bcs \calD^{\pm} } = \rtdcap{\calD} \# \calD_{\pm\CC{P^2}},
\end{equation*}
where $\calD_{\pm\CC{P^2}}$ denotes the standard $(1,0)$-trisection diagram of $\pm\CC{P^2}$.
\begin{figure}[!tbp]
\centering
\includegraphics[scale=0.7]{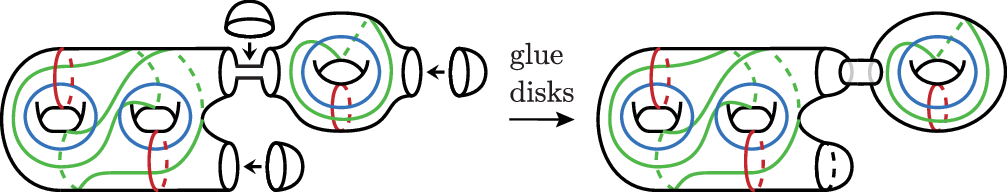}
\caption{$\rtdcap{ \calD \bcs \calD^{+} } = \rtdcap{\calD} \# \calD_{\CC{P^2}}$.}
\label{fig:td+cp2-gd}
\end{figure}
Therefore, the corresponding closed $4$-manifolds satisfy
\begin{equation*}
X( \rtdcap{\calD \bcs \calD^{\pm}}) \cong X( \rtdcap{\calD})\#(\pm\CC{P^2}).
\end{equation*}
Since $X(\rtdcap{\calD})\#\CC{P^2}$ and $X(\rtdcap{\calD})\#(-\CC{P^2})$ are not homeomorphic, Theorem~\ref{thm:equiv-rtd-gd} implies that the diagrams $\calD \bcs \calD^{+}$ and $\calD \bcs \calD^{-}$ are not  diffeomorphism and handleslide equivalent.
Hence, the relative trisections $\calRT(\calD \bcs \calD^{+})$ and $\calRT(\calD \bcs \calD^{-})$ are not diffeomorphic.
\end{proof}

\section{A geometric meaning of the capping operation}

In this section, we study how the underlying trisected $4$-manifold changes when disks are glued to the boundary of a relative trisection diagram.

\begin{lemma}\label{lem:rtdcap}
Let $\calD = (\Sigma; \alpha, \beta, \gamma)$ be a $(g, k; 0, b)$-relative trisection diagram.
\begin{itemize}
\item
Suppose that $b\geq2$.
Let $c \subset \partial \Sigma$ be a boundary component of the surface $\Sigma$.
Let $\calD' := (\Sigma\cup_{c}D^2; \alpha, \beta, \gamma)$ be the diagram obtained by gluing the disk $D^2$ to $\Sigma$ along $c$.
Then $\calD'$ is a $(g, k-1; 0, b-1)$-relative trisection diagram.
\item
Suppose that $b=1$.
Let $\calD' := (\Sigma\cup_{\partial}D^2; \alpha, \beta, \gamma)$ be the diagram obtained by gluing the disk $D^2$ to $\Sigma$ along the circle $\partial{\Sigma}$.
Then $\calD'$ is a $(g, k)$-trisection diagram.
\end{itemize}
\end{lemma}

The proof follows from the argument in the proof of Lemma~3.1 in \cite{Tak24}.

\begin{remark}
By Lemma~\ref{lem:rtdcap}, when the genus $p$ of the induced open book is $0$, the diagram $\rtdcap{\calD}$ obtained by capping off the all boundary components is always a trisection diagram of a closed $4$-manifold.
On the other hand, if $p\neq0$, the resulting diagram $\rtdcap{\calD}$ cannot be a trisection diagram.
Let $\calD$ be the $(2,3;1,2)$-relative trisection diagram shown in the left of Figure~\ref{fig:gdo-pnot0}.
\begin{figure}[!tbp]
\centering
\includegraphics[scale=0.6]{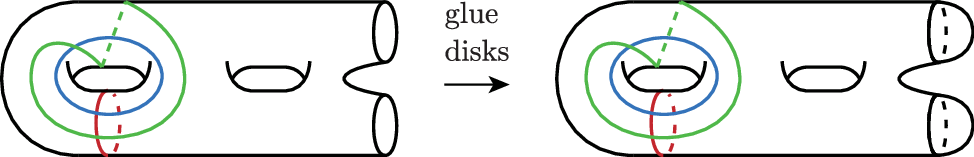}
\caption{The case of $p\neq0$.}
\label{fig:gdo-pnot0}
\end{figure}
By Theorem~\ref{thm:rt-prop}, $\calD$ induces an open book with page $\Sigma_{1,2}$, i.e. $p=1$.
The capped diagram $\rtdcap{\calD}$ is not a trisection diagram, since the number of components of each curve family is insufficient.
\end{remark}

For the case $b \ge 2$, we obtain the following.

\begin{theorem}\label{thm:cap-b>1}
Let $\calD = (\Sigma; \alpha, \beta, \gamma)$ be a $(g, k; 0, b)$-relative trisection diagram with $b \ge 2$.
Let $c \subset \partial \Sigma$ be a boundary component of the surface $\Sigma$.
Let $\calD' := (\Sigma \cup_{c} D; \alpha, \beta, \gamma)$ be the $(g, k-1; 0, b-1)$-relative trisection diagram obtained by gluing a disk $D$ to $\Sigma$ along $c$.
Let $\calRT(\calD): W = W_1 \cup W_2 \cup W_3$ be the $(g, k; 0, b)$-relative trisection induced by $\calD$, and let $\calRT(\calD'): W' = W_1' \cup W_2' \cup W_3'$ be the $(g, k-1; 0, b-1)$-relative trisection induced by $\calD'$.
Then the $4$-manifold $W'$ is obtained from $W$ by attaching a $2$-handle $h^{(2)}$ along the circle $c \subset \partial(W_{1} \cap W_{2} \cap W_{3})$. 
\end{theorem}

\begin{proof}
Since $\calRT(\calD') : W' = W_1' \cup W_2' \cup W_3'$ is a $(g, k-1; 0, b-1)$-relative trisection, Definition~\ref{dfn:rt} implies that for each $i \in \{1,2,3\}$ there exists a diffeomorphism
\begin{equation*}
\psi_i : W_i' \to Z_{k-1} = U_{0,b-1} \bcs V_n,
\end{equation*}
where $n = k - 2p - b + 1$.
The model $Z_{k-1}$ is given as in the left picture of Figure~\ref{fig:ZkZk-1}.
\begin{figure}[!tbp]
\centering
\includegraphics[scale=0.7]{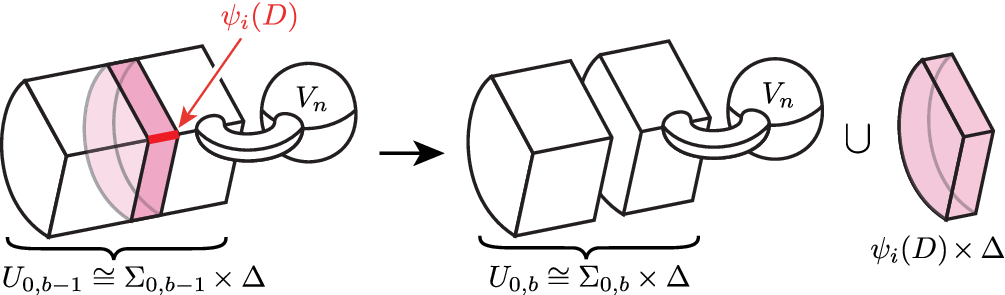}
\caption{$Z_{k-1} = Z_{k} \cup \left( \psi_i(D) \times \Delta \right)$.}
\label{fig:ZkZk-1}
\end{figure}
Since $\Sigma\cup_{c}D = W'_1 \cap W'_2 \cap W'_3$, we see that the disk $\psi_i(D)$ is embedded in $(Y^-_{g,k-1;0,b-1}\cap Y^+_{g,k-1;0,b-1})\cap U_{0,b-1} \cong\Sigma_{0,b-1}$.
Thus, $\psi_i(D)\times\Delta$ is a codimension $0$ submanifold of $U_{0,b-1}$ (see Figure~\ref{fig:ZkZk-1}).
Furthermore, we obtain the following decomposition:
\begin{equation*}
Z_{k-1} = Z_{k} \cup \left( \psi_i(D) \times \Delta \right).
\end{equation*}
This is derived as follows.
\begin{align*}
U_{0,b-1} &= \Sigma_{0,b-1} \times \Delta \\
&= \left( \Sigma_{0,b} \cup_{\psi_i(c)} \psi_{i}(D) \right) \times \Delta \\
&= \left( \Sigma_{0,b} \times \Delta \right) \cup_{\psi_i(c)\times\Delta} \left( \psi_i(D)\times\Delta \right) \\
&= U_{0,b} \cup_{\psi_i(c)\times\Delta} \left( \psi_i(D)\times \Delta \right)
\end{align*}
Therefore,
\begin{align*}
Z_{k-1} &= V_n \bcs U_{0,b-1} \\
&= V_n \bcs \left( U_{0,b} \cup_{ \psi_i(c)\times\Delta } \left( \psi_i(D)\times\Delta \right) \right) \\
&= \left( V_n \bcs U_{0,b} \right) \cup_{\psi_i(c)\times\Delta } \left( \psi_i(D)\times\Delta  \right) \\
&= Z_{k} \cup_{\psi_i(c)\times\Delta} \left( \psi_i(D)\times\Delta  \right)
\end{align*}
Note that $\psi_i(D)\times\Delta$ is a $2$-handle attached to $Z_k$.
Thus, we can regard $Z_k$ as a submanifold of $Z_{k-1}$, and hence consider the restrictions
\begin{equation*}
\psi_i^{-1}\big|_{Z_k} : Z_k \to \psi_i^{-1}(Z_k) \subset X_i.
\end{equation*}
The union of these images
\begin{equation*}
\psi_1^{-1}(Z_k) \cup \psi_2^{-1}(Z_k) \cup \psi_3^{-1}(Z_k)
\end{equation*}
is a $(g,k;0,b)$-relative trisection inducing $\calD = (\Sigma; \alpha, \beta, \gamma)$.
Thus, we can identify this decomposition with $\calRT(\calD) : W = W_1 \cup W_2 \cup W_3$.

Next, we set $h^{(2)}_i := \psi^{-1}_{i}(\psi_{i}(D)\times\Delta)$.
Then, we see that
\begin{align*}
W'_i &= \psi_i^{-1}(Z_{k-1}) \\
&= \psi_i^{-1}\left(Z_k \cup (\psi_{i}(D)\times\Delta) \right) \\
&= \psi_i^{-1}(Z_k) \cup \psi_i^{-1}\left(\psi_{i}(D)\times\Delta) \right) \\
&= \psi_i^{-1}(Z_k) \cup h^{(2)}_i.
\end{align*}
Finally, we define the $2$-handle $h^{(2)}$ as the union of these components:
\begin{align*}
h^{(2)} :=& h^{(2)}_1 \cup h^{(2)}_2 \cup h^{(2)}_3 \\
=&\, \psi^{-1}_{1}(\psi_{1}(D)\times\Delta) \cup \psi^{-1}_{2}(\psi_{2}(D)\times\Delta) \cup \psi^{-1}_{3}(\psi_{3}(D)\times\Delta).
\end{align*}
This decomposition is a $(0,0;0,1)$-relative trisection of $D^2\times D^2$, and the corresponding diagram is $(D;\emptyset,\emptyset,\emptyset)$.
Then we see that $W'=W \cup h^{(2)}$, since
\begin{align*}
W' &= W'_1 \cup W'_2 \cup W'_3 \\
&= \left(\psi_1^{-1}(Z_k) \cup h^{(2)}_1 \right) \cup \left(\psi_2^{-1}(Z_k) \cup h^{(2)}_2 \right) \cup \left(\psi_3^{-1}(Z_k) \cup h^{(2)}_3 \right) \\
&= \left(\psi_1^{-1}(Z_k) \cup \psi_2^{-1}(Z_k) \cup \psi_3^{-1}(Z_k) \right) \cup \left(h^{(2)}_1 \cup h^{(2)}_2 \cup h^{(2)}_3 \right) \\
&= (W_1 \cup W_2 \cup W_3) \cup h^{(2)} \\
&= W \cup h^{(2)}.
\end{align*}
See Figure~\ref{fig:glue2hdl}, for a schematic picture.
The attaching region of $h^{(2)}$ is the union of the attaching regions of $h_1^{(2)}$, $h_2^{(2)}$, and $h_3^{(2)}$, that is, 
\begin{align*}
\psi^{-1}_{1}(\psi_{1}(c)\times\Delta) \cup \psi^{-1}_{2}(\psi_{2}(c)\times\Delta) \cup \psi^{-1}_{3}(\psi_{3}(c)\times\Delta).
\end{align*}
The attaching circle of $h^{(2)}$ is given by
\begin{align*}
\partial(h^{(2)}_1 \cap h^{(2)}_2 \cap h^{(2)}_3)
= \psi^{-1}_{i}(\psi_{i}(c)\times\{0\}).
\end{align*}
We can choose the framing of $h^{(2)}$ as $\psi^{-1}_{i}(\psi_{i}(c)\times\{e^{\pm\sqrt{-1}\pi/3}\})$.
That is, when we identify the compression body $X_i \cap X_j$
with $(\Sigma \times [0,1]) \cup (\text{$2$-handles})$,
the curve $c \times \{0\}$ is the attaching circle,
and $c \times \{1\}$ gives the framing.
\end{proof}
\begin{figure}[!tbp]
\centering
\includegraphics[scale=0.7]{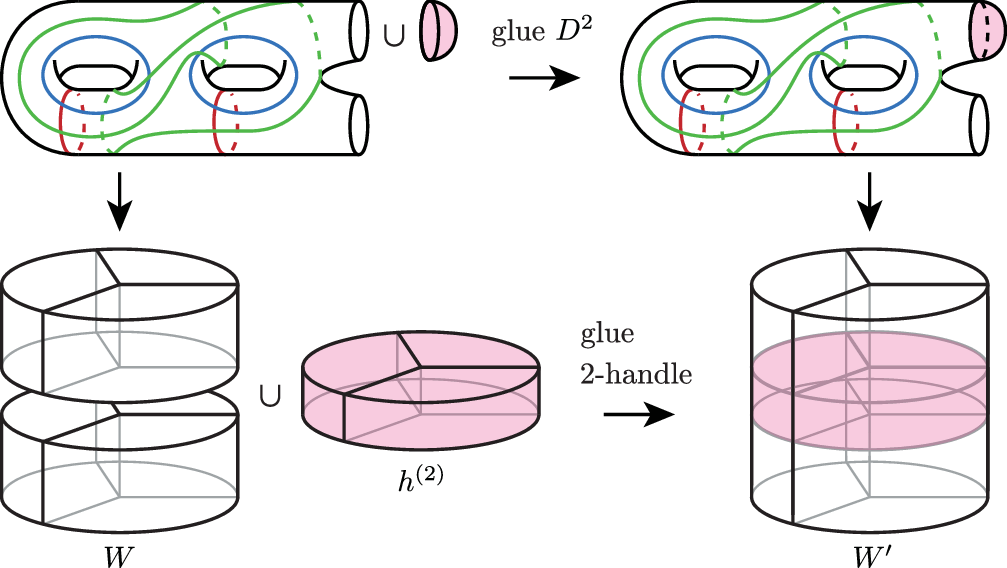}
\caption{Correspondence between a gluing of $D^2$ to a relative trisection diagram and the induced change in the trisected $4$-manifold.}
\label{fig:glue2hdl}
\end{figure}

For the case $b=1$, we obtain the following. 

\begin{theorem}\label{thm:cap-b=1}
Let $\calD = (\Sigma; \alpha, \beta, \gamma)$ be a $(g, k; 0, 1)$-relative trisection diagram.
Let $\calD' := (\Sigma\cup_{\partial}D^2; \alpha, \beta, \gamma)$ be the $(g, k)$-trisection diagram obtained by gluing the disk $D^2$ to the surface $\Sigma$ along the boundary circle.
Let $\calRT(\calD): W = W_1 \cup W_2 \cup W_3$ be the $(g, k; 0, 1)$-relative trisection induced by $\calD$, and let $\calT(\calD'): X = X_1 \cup X_2 \cup X_3$ be the $(g, k)$-trisection induced by $\calD'$.
Then the closed $4$-manifold $X$ is obtained from $W$ by attaching a $4$-handle.
\end{theorem}

\begin{proof}
Gluing a disk $D^2$ to $\calD$ can be regarded as attaching the $(0,0;0,1)$-relative trisection diagram $(D^2;\emptyset,\emptyset,\emptyset)$ of the $4$-ball $D^4$.
This is because the open books induced by $\calD$ and $(D^2;\emptyset,\emptyset,\emptyset)$ both agree with $(D^2,\id)$, and hence can be glued as relative trisection diagrams by Proposition~2.12 of \cite{CasOzb19}.
Since $\calD'=\calD \cup_{\partial} (D^2;\emptyset,\emptyset,\emptyset)$, it follows that $W'\cong W\cup_{S^3} D^4$.
\end{proof}

Combining Theorems~\ref{thm:cap-b>1} and \ref{thm:cap-b=1}, we obtain Theorem~\ref{thm:cap-handle}.

\section{Minimal genus relative trisections with distinct parameters}

In this section, we present examples of minimal genus relative trisections of the same $4$-manifold whose associated parameters $(g,k;p,b)$ are different.
Note that relative trisections with different parameters $(g,k;p,b)$ are necessarily non-diffeomorphic.

\begin{proof}[Proof of Theorem~\ref{thm:diff-gkpb}]
Let $W^{-}(0,1)$ be the contractible $4$-manifold given by the handlebody diagram in Figure~\ref{fig:Kd-W-01}.
\begin{figure}[!tbp]
\centering
\includegraphics[scale=0.7]{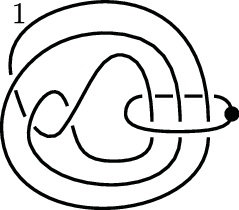}
\caption{A handlebody diagram of $W^{-}(0,1)$.}
\label{fig:Kd-W-01}
\end{figure}
The trisection genus of $W^{-}(0,1)$ is $3$ (see \cite[Theorem~1.3]{Tak25b}).
Let $\calD$ and $\calD'$ be the diagrams depicted on the left and right of Figure~\ref{fig:td-W-01}, respectively.
\begin{figure}[!tbp]
\centering
\includegraphics[scale=0.7]{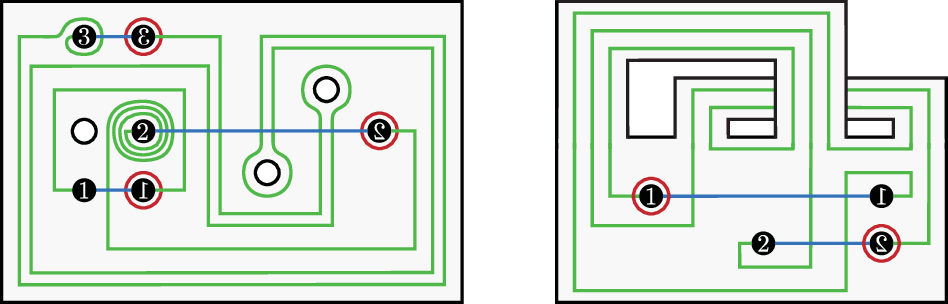}
\caption{Two genus-$3$ relative trisection diagrams of $W^{-}(0,1)$.}
\label{fig:td-W-01}
\end{figure}
Here, $\calD$ is a $(3,3;0,4)$-relative trisection diagram of $W^{-}(0,1)$ (see the proof of Theorem~1.3 in \cite{Tak25b}).
We now show that $\calD'$ is a $(3,2;1,1)$-relative trisection diagram of $W^{-}(0,1)$.
We want to verify that $(\Sigma;\alpha,\beta)$, $(\Sigma;\beta,\gamma)$, and $(\Sigma;\gamma,\alpha)$ are diffeomorphism and handleslide equivalent to the standard diagram in Figure~\ref{fig:td-std3211}.
\begin{figure}[!tbp]
\centering
\includegraphics[scale=0.7]{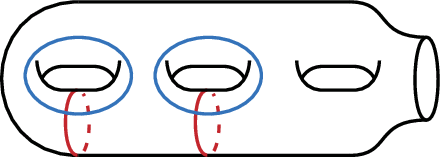}
\caption{Standard sutured Heegaard diagram of type $(3,2;1,1)$.}
\label{fig:td-std3211}
\end{figure}
It is easy to see that $(\Sigma;\alpha,\beta)$ and $(\Sigma;\gamma,\alpha)$ are standard (see Figures~\ref{fig:td-Sab} and \ref{fig:td-Sca}).
\begin{figure}[!tbp]
  \begin{minipage}[b]{0.49\linewidth}
    \centering
    \includegraphics[scale=0.65]{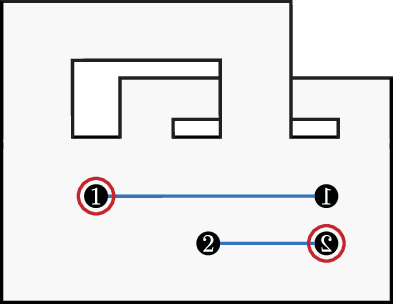}
    \caption{$(\Sigma;\alpha,\beta)$}
    \label{fig:td-Sab}
  \end{minipage}
  \begin{minipage}[b]{0.49\linewidth}
    \centering
    \includegraphics[scale=0.65]{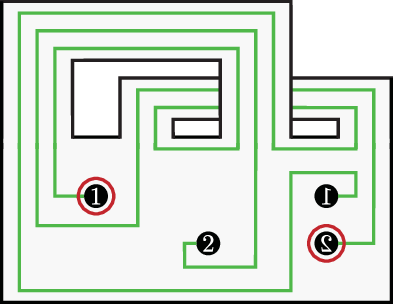}
    \caption{$(\Sigma;\gamma,\alpha)$}
    \label{fig:td-Sca}
  \end{minipage}
\end{figure}
As shown in Figure~\ref{fig:td-Sbc-modif}, we can modify $(\Sigma;\gamma,\alpha)$ into the standard diagram by diffeomorphisms on $\Sigma$ and slides of the curves.
\begin{figure}[!tbp]
\centering
\includegraphics[scale=0.7]{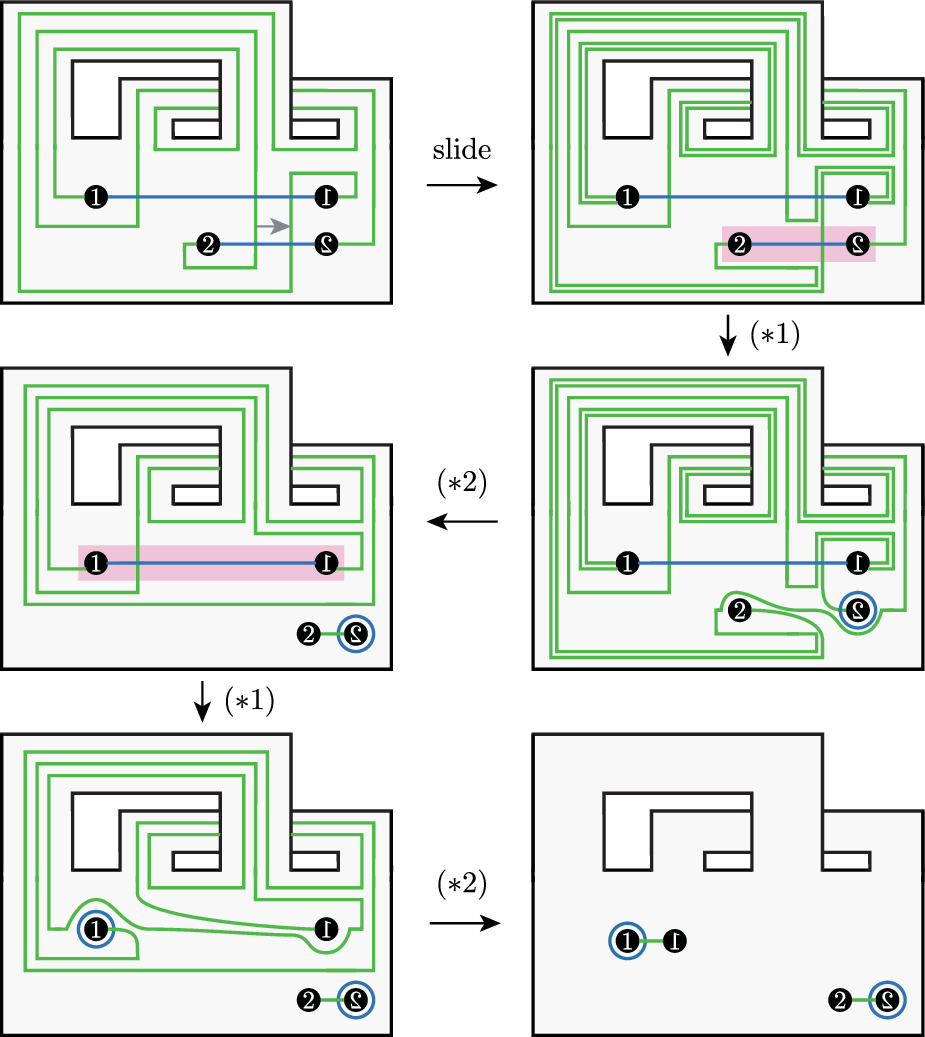}
\caption{Diffeomorphisms and handleslides that modify $(\Sigma;\beta,\gamma)$ into the standard diagram.}
\label{fig:td-Sbc-modif}
\end{figure}
Note that the operation $(*1)$ is a diffeomorphism acting on the shaded region.
The operation $(*2)$ consists of dragging the black disk together with the $\beta$-circle.
When the black disk approaches another $\beta$-curve, then it can pass through by performing a slide of the $\beta$-curve.
For details of these operations, see Subsection~2.3 of \cite{Tak25a}.

By applying the algorithm given in Subsection 2.2 of \cite{KimMil20}, the relative trisection diagram $\calD'$ yields the first handlebody diagram of Figure~\ref{fig:Kd-move-W-01}.
The subsequent handlemoves in the same figure show that the $4$-manifold induced by $\calD'$ is diffeomorphic to $W^{-}(0,1)$.
\begin{figure}[!tbp]
\centering
\includegraphics[scale=0.7]{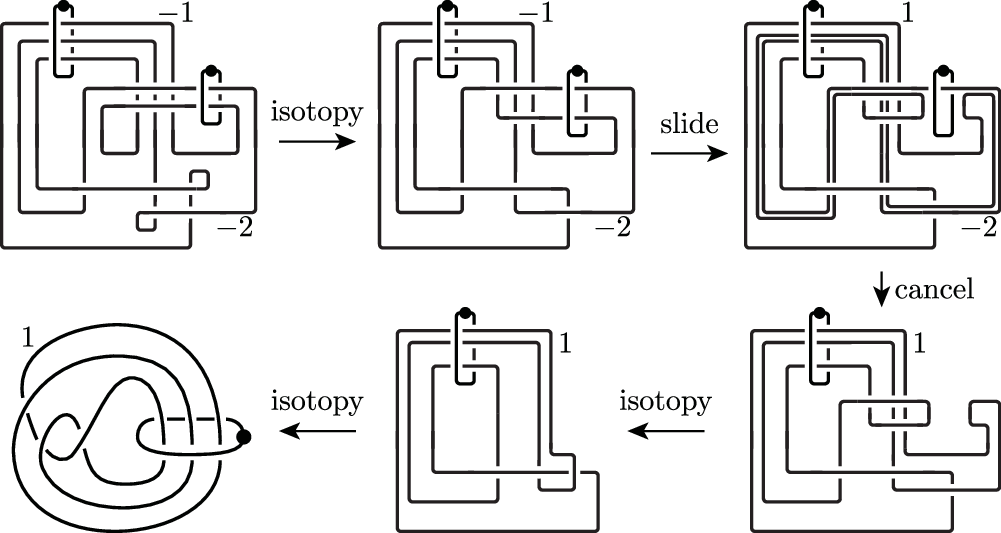}
\caption{Handlemoves of $W^{-}(0,1)$.}
\label{fig:Kd-move-W-01}
\end{figure}
\end{proof}

\section*{Appendix~A.~Connected sum operation of a trisection and a relative trisection}

In this appendix, we describe the connected sum operation for a trisection of a closed $4$-manifold and a relative trisection of a $4$-manifold with boundary.
Although the construction is essentially the same as the connected sum of two closed trisections, to the best of our knowledge, a detailed description does not appear explicitly in the literature.

To define the connected sum of trisections, we remove a small $4$-ball from each trisected $4$-manifold.
This $4$-ball is required to contain a small disk in the triple intersection surface and to be compatible with the given trisection.
First, we provide a detailed description of how to properly choose such a $4$-ball.
Let $X = X_1 \cup X_2 \cup X_3$ be a (relative) trisection, and let $\Sigma := X_1 \cap X_2 \cap X_3$ be the triple intersection surface.
Let $D \subset \Sigma$ be a smoothly embedded disk.
We define a neighborhood $\nu(D; X)$ of $D$ in $X$ as follows.

\begin{enumerate}[label=(\arabic*)]
\item
We define $\nu(D; X_i \cap X_j)$ as a collar neighborhood of $D$ in the $3$-manifold $X_i \cap X_j$. Then we have
\begin{equation*}
\nu(D; X_i \cap X_j) \cong D \times [0, 1] \cong D^3.
\end{equation*}
\item
We define $\nu(D; \partial X_i)$ as the union of the two collar neighborhoods $\nu(D; X_i \cap X_{i\pm1})$ in the boundary $3$-manifold $(X_i \cap X_{i+1}) \cup(X_i \cap X_{i-1})$:
\begin{equation*}
\nu(D; \partial X_i) := \nu(D; X_i \cap X_{i+1}) \cup \nu(D; X_i \cap X_{i-1}).
\end{equation*}
Since the intersection is $D$, we see that $\nu(D; \partial X_i) \cong D^3$.
\item
We define $\nu(D; X_i)$ as a collar neighborhood of $\nu(D; \partial X_i)$ in the $4$-manifold $X_i$:
\begin{equation*}
\nu(D; X_i) := \nu(\nu(D; \partial X_i); X_i).
\end{equation*}
This $4$-manifold is diffeomorphic $D^3\times[0,1]\cong D^4$. See Figure~\ref{fig:nuD}.
\item
Finally, we define $\nu(D; X)$ as the union of the three neighborhoods:
\begin{equation*}
\nu(D; X) := \nu(D; X_1) \cup \nu(D; X_2) \cup \nu(D; X_3).
\end{equation*}
\end{enumerate}

This decomposition is a $(0,0;0,1)$-relative trisection of $D^4$, and the corresponding diagram is $(D;\emptyset,\emptyset,\emptyset)$.
\begin{figure}[!tbp]
\centering
\includegraphics[scale=0.7]{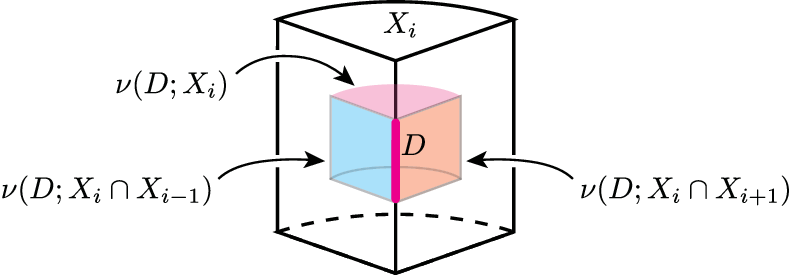}
\caption{A schematic picture of $\nu(D;X_i)$.}
\label{fig:nuD}
\end{figure}

Now we are ready to define the connected sum for a trisection of a closed $4$-manifold and a relative trisection of a $4$-manifold with boundary.

\begin{definition}\label{dfn:consum:tr-rt}
Let $\calT:X = X_1 \cup X_2 \cup X_3$ be a $(g,k)$-trisection of a closed $4$-manifold $X$, and let $\calRT':W' = W'_1 \cup W'_2 \cup W'_3$ be a $(g',k';p',b')$-relative trisection of a $4$-manifold $W'$ with boundary.
Let $D \subset X_1 \cap X_2 \cap X_3$ and $D' \subset W'_1 \cap W'_2 \cap W'_3$ be small disks in the respective triple intersection  surfaces.
Choose neighborhoods $\nu(D;X)$ and $\nu(D';W')$ as described above.
We define the connected sum $X \# W'$ by
\begin{equation*}
X\# W' := ( X-\nu(D;X) ) \cup ( W'-\nu(D';W') )
\end{equation*}
where the two boundary $2$-spheres $\partial \nu(D;X)$ and $\partial \nu(D';W')$ are identified.
For each $i \in \{1,2,3\}$, the new sector $W''_i$, which is a $4$-dimensional $1$-handlebody of genus $k+k'$, is defined by
\begin{equation*}
W''_{i} := ( X_{i}-\nu(D;X) ) \bcs ( W'_{i}-\nu(D';W') )
\end{equation*}
where the boundary connected sum is taken by identifying the two $3$-balls $X_i \cap \partial \nu(D;X)$ and $W'_i \cap \partial \nu(D';W')$.
The $(g+g',k+k';p',b')$-relative trisection of $X \# W'$ obtained as the connected sum of $\calT$ and $\calRT'$ is defined by
\begin{equation*}
\calT \# \calRT' : X\# W = W''_{1} \cup W''_{2} \cup W''_{3}.
\end{equation*}
\end{definition}

A relative trisection diagram representing the connected sum $\calT \# \calRT'$ is given as follows.
Let $\calD := (\Sigma;\alpha,\beta,\gamma)$ be a $(g,k)$-trisection diagram of $\calT$,
and let $\calD' := (\Sigma';\alpha',\beta',\gamma')$ be a $(g',k';p',b')$-relative trisection diagram of $\calRT'$.
Then the diagram
\begin{equation*}
\calD\#\calD':=(\Sigma\#\Sigma';\alpha\cup\alpha',\beta\cup\beta',\gamma\cup\gamma')
\end{equation*}
is a $(g+g',k+k';p',b')$-relative trisection diagram of $\calT \# \calRT'$.
Here, the connected sum $\Sigma \# \Sigma'$ is defined by removing the disks $D \subset \Sigma$ and $D' \subset \Sigma'$ as in Definition~\ref{dfn:consum:tr-rt} and identifying their boundaries:
\begin{equation*}
\Sigma \# \Sigma' := (\Sigma-\Int{D}) \cup_{\partial{D}\sim\partial{D'}} (\Sigma'-\Int{D'}).
\end{equation*}


\section*{Acknowledgements}

The author would like to thank Kouichi Yasui for helpful comments.
The author was partially supported by JSPS KAKENHI Grant Number 24KJ1561.


\begin{thebibliography}{CGPC18}

\bibitem[Cas16]{Cas16}
Nickolas~A. Castro, \emph{Relative trisections of smooth 4-manifolds with
  boundary}, Ph.D. thesis, University of Georgia, 2016.

\bibitem[CGPC18]{CasGayPin18a}
Nickolas~A. Castro, David~T. Gay, and Juanita Pinz\'{o}n-Caicedo,
  \emph{Diagrams for relative trisections}, Pacific J. Math. \textbf{294}
  (2018), no.~2, 275--305.

\bibitem[CO19]{CasOzb19}
Nickolas~A. Castro and Burak Ozbagci, \emph{Trisections of 4-manifolds via
  {L}efschetz fibrations}, Math. Res. Lett. \textbf{26} (2019), no.~2,
  383--420.

\bibitem[EO08]{EtnOzb08}
John Etnyre and Burak Ozbagci, \emph{Invariants of contact structures from open
  books}, Transactions of the American Mathematical Society \textbf{360}
  (2008), no.~6, 3133--3151.

\bibitem[GG06]{GirGoo06}
Emmanuel Giroux and Noah Goodman, \emph{On the stable equivalence of open books
  in three-manifolds}, Geom. Topol. \textbf{10} (2006), 97--114 (English).

\bibitem[GK16]{GayKir16}
David Gay and Robion Kirby, \emph{Trisecting 4-manifolds}, Geom. Topol.
  \textbf{20} (2016), no.~6, 3097--3132.

\bibitem[GS99]{GomSti99}
Robert~E. Gompf and Andr\'{a}s~I. Stipsicz, \emph{{$4$}-manifolds and {K}irby
  calculus}, Graduate Studies in Mathematics, vol.~20, American Mathematical
  Society, Providence, RI, 1999.

\bibitem[Isl21]{Isl21}
Gabriel Islambouli, \emph{Nielsen equivalence and trisections}, Geom. Dedicata
  \textbf{214} (2021), 303--317.

\bibitem[IO24]{IsoOga24a}
Tsukasa Isoshima and Masaki Ogawa, \emph{Nielsen equivalence and multisections of {$4$}-manifolds}, preprint, arXiv:2401.15822.

\bibitem[KM20]{KimMil20}
Seungwon Kim and Maggie Miller, \emph{Trisections of surface complements and
  the {P}rice twist}, Algebr. Geom. Topol. \textbf{20} (2020), no.~1, 343--373.

\bibitem[Rol76]{Rol76}
Dale Rolfsen, \emph{Knots and links}, Mathematics Lecture Series, vol. No. 7,
  Publish or Perish, Inc., Berkeley, CA, 1976.

\bibitem[Tak24]{Tak24}
Natsuya Takahashi, \emph{Non-diffeomorphic minimal genus relative trisections
  of the same 4-manifold}, J. Knot Theory Ramifications \textbf{33} (2024),
  no.~6, Paper No. 2450016. 

\bibitem[Tak25a]{Tak25a}
Natsuya Takahashi, \emph{Minimal genus relative trisections of corks}, Geom. Dedicata
  \textbf{219} (2025), no.~2, Paper No. 29. 

\bibitem[Tak25b]{Tak25b}
Natsuya Takahashi, \emph{Exotic 4-manifolds with small trisection genus}, Topology Appl.
  \textbf{368} (2025), Paper No. 109382. 

\bibitem[Wal68]{Wald68}
Friedhelm Waldhausen, \emph{Heegaard-{Z}erlegungen der {$3$}-{S}ph\"{a}re},
  Topology \textbf{7} (1968), 195--203.

\end{thebibliography}



\end{document}